\documentclass[11pt]{article}

\author{Joshua E. Rooney}
\date{May 28, 2025} 
\title{Winning Probabilities of Balanced and Nontransitive $n$-tuples of Dice}
 
\usepackage{graphicx}

\usepackage{amsthm}
\usepackage{amsmath}
\usepackage{amsfonts}
\usepackage[bottom]{footmisc}
\usepackage{todonotes}
\usepackage{verbatim}

\usepackage[margin=1in]{geometry}

\usepackage[shortlabels]{enumitem}

\usepackage{tikz}
\usepackage{setspace}
\usepackage{hyperref}

\usepackage[T1,T2A]{fontenc}
\usepackage[russian, english]{babel}
\usepackage{csquotes}

\usepackage{color}

\newtheorem{theorem}{Theorem}[section]
\newtheorem{problem}[theorem]{Problem}

\newtheorem{lemma}[theorem]{Lemma}

\theoremstyle{definition}
\newtheorem{defn}[theorem]{Definition}

\newtheorem{observation}[theorem]{Observation}
\newtheorem*{remark}{Remark}

\usepackage{fancyhdr}
\begin{document}

\maketitle

\begin{abstract}
    For a positive integer $n$, an $n$-tuple of dice $(A_1,A_2,\dots,A_n)$ is called \emph{balanced} if $P(A_1<A_2) = P(A_2<A_3) = \cdots = P(A_n<A_1)$ and \emph{nontransitive} if $P(A_1<A_2), P(A_2<A_3), \dots, P(A_n<A_1)$ are each greater than $\frac{1}{2}$. For a balanced and nontransitive $n$-tuple of dice $(A_1,A_2,\dots,A_n)$, we define the winning probability $w(A_1,A_2,\dots,A_n) := P(A_1 < A_2)$. The works of Trybula and Kim et al. together show that for a balanced and nontransitve triple of dice $(A_1,A_2,A_3)$, the least upper bound on the winning probability is $\frac{-1+\sqrt{5}}{2}$. Kim et al. then asked what the least upper bound on the winning probability was for the $n \geq 4$ cases. Bogdanov and Komisarski independently have shown that for $n\geq 3$ and a balanced and nontransitive $n$-tuple of dice $(A_1,A_2,\dots,A_n)$, the winning probability is less than $\pi_n := 1-\frac{1}{4\cos^2\left( \frac{\pi}{n+2} \right)}$.
    
    In this paper, we will show that for $n \geq 3$ and every rational $p \in  \left( \frac{1}{2}, \pi_n \right]$, there exists a balanced and nontransitive $n$-tuple of dice with winning probability $p$. Paired with Bogdanov and Komisarski's results, this fully answers the problem posed by Kim et al. and establishes a complete characterization of the winning probabilities for nontransitive and balanced $n$-tuples of dice.
\end{abstract}

\section{Introduction}

A die $A_1$ is a discrete uniform random variable with finite support in the positive integers. For dice $A_1$ and $A_2$, let $P(A_1<A_2)$ be the probability that $A_2$ is greater than $A_1$.
\begin{defn}
    For an integer $n \geq 3$, an $n$-tuple of dice $(A_1,A_2,\dots, A_n)$ is called
    \begin{itemize}
        \item \emph{balanced} if $P(A_1 < A_2) = P(A_2 < A_3) = \cdots = P(A_n < A_1)$;
        \item \emph{nontransitive} if, for each $i \in \{ 1,\dots,n-1\}$, we have $P(A_i < A_{i+1}) > \frac{1}{2}$, and $P(A_n < A_1) > \frac{1}{2}$.
    \end{itemize}
\end{defn}
Popularized by Martin Gardner \cite{Gardner} in 1970, the Efron dice shown in Figure \ref{EfronDice} are a classic example of a balanced and nontransitive $4$-tuple of dice $(A_1,A_2,A_3,A_4)$.
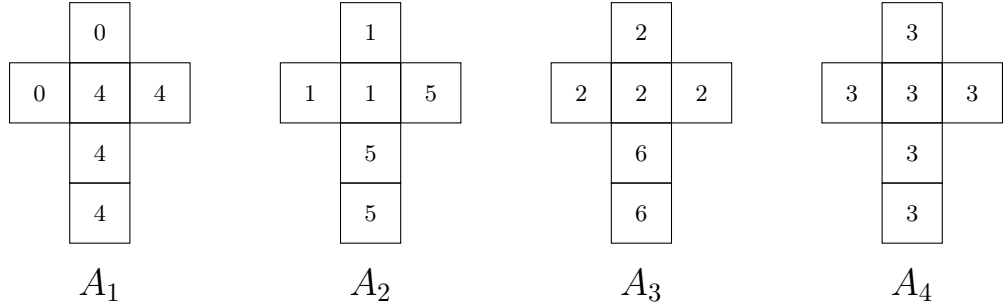
\begin{figure}[ht]
\centering
\begin{tikzpicture}[every node/.style={draw, minimum size=0.8cm, font=\footnotesize}, scale=0.8]

\node at (0,0) {4};
\node at (0,1) {0};
\node at (0,-1) {4};
\node at (-1,0) {0};
\node at (1,0) {4};
\node at (0,-2) {4};
\node[draw=none, font=\Large\bfseries] at (0,-3.2) {$A_1$};

\begin{scope}[xshift=4.5cm]
  \node at (0,0) {1};
  \node at (0,1) {1};
  \node at (0,-1) {5};
  \node at (-1,0) {1};
  \node at (1,0) {5};
  \node at (0,-2) {5};
  \node[draw=none, font=\Large\bfseries] at (0,-3.2) {$A_2$};
\end{scope}

\begin{scope}[xshift=9cm]
  \node at (0,0) {2};
  \node at (0,1) {2};
  \node at (0,-1) {6};
  \node at (-1,0) {2};
  \node at (1,0) {2};
  \node at (0,-2) {6};
  \node[draw=none, font=\Large\bfseries] at (0,-3.2) {$A_3$};
\end{scope}

\begin{scope}[xshift=13.5cm]
  \node at (0,0) {3};
  \node at (0,1) {3};
  \node at (0,-1) {3};
  \node at (-1,0) {3};
  \node at (1,0) {3};
  \node at (0,-2) {3};
  \node[draw=none, font=\Large\bfseries] at (0,-3.2) {$A_4$};
\end{scope}

\end{tikzpicture}

\caption{Efron Dice; $P(A_1 < A_2) = P(A_2 < A_3) = P(A_3 < A_4) = P(A_4 < A_1) = \frac{2}{3}$} \label{EfronDice}
\end{figure}

Adopting terminology from \cite{Kim2025Balanced4Dice}, for a balanced and nontransitive (BN) $n$-tuple of dice \\$(A_1,A_2,\dots,A_n)$, we define the \emph{winning probability}
$$w(A_1,A_2,\dots,A_n) := P(A_1,<A_2) = P(A_2<A_3) = \cdots = P(A_n < A_1) > \frac{1}{2},$$
which is well-defined because this equality is assumed in the balanced case. 

Building on the $n=3$ case resolved by D. Kim, R. Kim, Lee, Lim, and So \cite{KKLLS2024} and Trybula \cite{Trybula61}, this paper proves the following result on the possible winning probabilities for a BN $n$-tuple of dice.
\begin{theorem} \label{mainThm}
    Define
    \begin{equation}\label{Pi_nDef}
        \pi_n := 1-\frac{1}{4\cos^2\left( \frac{\pi}{n+2} \right)}.
    \end{equation}
    Then, for every rational $p \in \left( \frac{1}{2}, \pi_n \right]$, there exists a positive integer $m$ and a BN $n$-tuple of dice $(A_1,A_2,\dots, A_n)$, each die with $m$ sides, such that
    $$w(A_1,A_2,\dots,A_n) = p.$$
\end{theorem}
Paired with Theorem \ref{Komisarski} (stated below), this fully answers Problem \ref{KKLLSProb} (stated below) and establishes a complete characterization of the winning probabilities for nontransitive and balanced $n$-tuples of dice.

In 1961, Stanisław Trybula \cite{Trybula61} found that the winning probability for a BN triple of dice was at most $\frac{-1+\sqrt{5}}{2}$. In 2017, Schaefer and Schweig \cite{SchaeferSchweig2017} asked whether a better bound could be found. Hur and Y. Kim \cite{HurKim} then conjectured that this upper bound could be reduced to $\frac{1}{2} + \frac{1}{9}$. In 2024, D. Kim, R. Kim, Lee, Lim, and So \cite{KKLLS2024} disproved this conjecture with the following theorem, fully classifying the winning probabilities for BN triples of dice.
\begin{theorem}\label{KKLLSThm}
    \cite{KKLLS2024} For each rational $p \in \left( \frac{1}{2}, \frac{-1+\sqrt{5}}{2} \right]$, there exists a BN triple of dice with winning probability $p$.
\end{theorem}
Further, D. Kim, R. Kim, Lee, Lim, and So \cite{KKLLS2024} posed the following extension to their work.
\begin{problem}\label{KKLLSProb}
    \cite{KKLLS2024} For an integer $n \geq 4$, if $(A_1,A_2,\dots,A_n)$ is a BN $n$-tuple of dice, what is the least upper bound of the winning probability $w(A_1,A_2,\dots,A_n)$?
\end{problem}

In 1965, Trybula \cite{Trybula65} classified an upper bound on the winning probability of a BN $n$-tuple of dice implicitly. Then, in 2010, Bogdanov \cite{Bogdanov2010} found a closed-form expression for an upper bound on this winning probability, published in Russian. Further, in 2021, Andrzej Komisarski \cite{Komisarski2021} confirmed the result of Bogdanov using a distinct approach, which is interestingly geometric.
\begin{theorem}\label{Komisarski}
    \cite{Bogdanov2010, Komisarski2021} Let $(A_1,A_2,\dots, A_n)$ be a BN $n$-tuple of dice. Then, $w(A_1,A_2,\dots,A_n) \leq \pi_n$.
\end{theorem}

In this paper, we will extend the result of D. Kim, R. Kim, Lee, Lim, and So to $n \geq 4$ dice with Theorem \ref{mainThm}. Together with Theorem \ref{Komisarski}, this will fully classify the possible winning probabilities for a balanced and nontransitive $n$-tuple of dice, answering Problem \ref{KKLLSProb}.

We will proceed as follows. In Section \ref{Preliminaries}, we introduce preliminaries and set conventions. In Section \ref{MainProof}, we will prove Theorem \ref{mainThm}. Finally, in Section \ref{LemmaProofSection}, we will prove Lemmas \ref{mainLem1} and \ref{RationalExistenceLem} used in the proof of Theorem $\ref{mainThm}$, including the statement and proof of Lemmas \ref{L(*)<1/2} and \ref{U(*)=Pi_n}.

\section{Preliminaries}\label{Preliminaries}

Violating the transitive property, the mere existence of nontransitive $n$-tuples of dice is fascinating. As a result, sets of nontransitive dice have been heavily studied, see \cite{AngelDavis,BGH,kronenthal2009tied, savage1994paradox} and further sources therein.

Then came the study of balanced $n$-tuples of dice. First seen in \cite{SchaeferSchweig2017}, Schaefer and Schweig looked at $n$-tuples of dice that were both balanced and nontransitive. Since then, balanced and nontransitive dice have been further studied in \cite{HurKim, joyce2020strongly, schaefer2011balanced, SchaeferSolo2017}, with particular focus on triples and $4$-tuples of dice.

Although Efron dice were popularized in 1970, Hugo Steinhaus and Stanislaw Trybula \cite{SteinhausTrybula1959} had been studying nontransitive random variables since 1959. Following conventions set by Steinhaus and Trybula, we shall use distinct, consecutive positive integers on each face across all dice, starting from $1$. Under these conventions, the Efron dice shown in Figure \ref{EfronDice} would become the dice shown in Figure \ref{ModEfronDice}. Notice that since the relative ordering of the dice is maintained, $P(A_1<A_2), P(A_2<A_3),P(A_3<A_4), \text{ and } P(A_4<A_1)$ remain the same.

\begin{figure}[ht]
\centering
\begin{tikzpicture}[every node/.style={draw, minimum size=0.8cm, font=\footnotesize}, scale=0.8]

\node at (0,0) {16};
\node at (0,1) {1};
\node at (0,-1) {18};
\node at (-1,0) {2};
\node at (1,0) {17};
\node at (0,-2) {19};
\node[draw=none, font=\Large\bfseries] at (0,-3.2) {$A_1$};

\begin{scope}[xshift=4.5cm]
  \node at (0,0) {5};
  \node at (0,1) {3};
  \node at (0,-1) {21};
  \node at (-1,0) {4};
  \node at (1,0) {20};
  \node at (0,-2) {22};
  \node[draw=none, font=\Large\bfseries] at (0,-3.2) {$A_2$};
\end{scope}

\begin{scope}[xshift=9cm]
  \node at (0,0) {8};
  \node at (0,1) {6};
  \node at (0,-1) {23};
  \node at (-1,0) {7};
  \node at (1,0) {9};
  \node at (0,-2) {24};
  \node[draw=none, font=\Large\bfseries] at (0,-3.2) {$A_3$};
\end{scope}

\begin{scope}[xshift=13.5cm]
  \node at (0,0) {12};
  \node at (0,1) {10};
  \node at (0,-1) {14};
  \node at (-1,0) {11};
  \node at (1,0) {13};
  \node at (0,-2) {15};
  \node[draw=none, font=\Large\bfseries] at (0,-3.2) {$A_4$};
\end{scope}

\end{tikzpicture}

\caption{Modified Efron Dice}\label{ModEfronDice}
\end{figure}
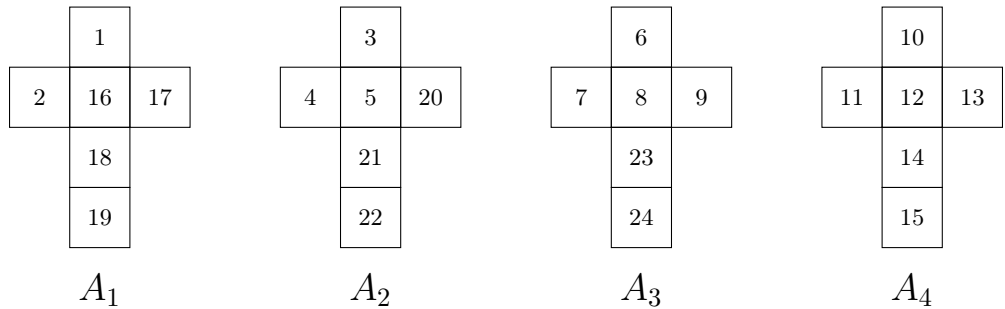

Under these conventions, it becomes clear that relative ordering is truly what determines $P(A_1<A_2),\dots,P(A_n<A_1)$. Say, for example, that we doubled every integer in Figure \ref{ModEfronDice}, $P(A_1<A_2),P(A_2<A_3), P(A_3<A_4), \text{ and } P(A_4<A_1)$ would remain the same. Accordingly, for ease, we represent sets of dice as words.

\begin{defn}
    For an $n$-tuple of dice ${\cal O} = (A_1,A_2, \dots, A_n)$, a \emph{word} $\sigma = \sigma(\cal O)$ is a sequence of letters, each of which is $A_1,\dots, A_{n-1}$ or $A_n$ such that the $i$th letter in the word is the die that contains the number $i$. Further, $\sigma$ is a $(m_1,\dots,m_n)$-word if $A_i$ has $m_i$ sides for all $i \in \{1,\dots,n \}$.
\end{defn}

For example, using the triple of dice in Figure \ref{1stBalNontranEx},
$$\sigma(A_1,A_2,A_3) = A_3A_3A_1A_1A_2A_2A_2A_2A_3A_3A_1A_1A_1A_1A_2A_2A_3A_3 = A_3^2A_1^2A_2^4A_3^2A_1^4A_2^2A_3^2.$$
Similarly, for the triple of dice in Figure \ref{2ndBalNontranEx},
$$\sigma(A_1,A_2,A_3) = A_1A_2A_1A_2^2A_3^4A_1^3A_2^2.$$

\begin{figure}[ht]
\centering
\begin{tikzpicture}[every node/.style={draw, minimum size=0.8cm, font=\footnotesize}, scale=0.8]

\node at (0,0) {11};
\node at (0,1) {3};
\node at (0,-1) {13};
\node at (-1,0) {4};
\node at (1,0) {12};
\node at (0,-2) {14};
\node[draw=none, font=\Large\bfseries] at (0,-3.2) {$A_1$};

\begin{scope}[xshift=5.5cm]
  \node at (0,0) {7};
  \node at (0,1) {5};
  \node at (0,-1) {15};
  \node at (-1,0) {6};
  \node at (1,0) {8};
  \node at (0,-2) {16};
  \node[draw=none, font=\Large\bfseries] at (0,-3.2) {$A_2$};
\end{scope}

\begin{scope}[xshift=11cm]
  \node at (0,0) {9};
  \node at (0,1) {1};
  \node at (0,-1) {17};
  \node at (-1,0) {2};
  \node at (1,0) {10};
  \node at (0,-2) {18};
  \node[draw=none, font=\Large\bfseries] at (0,-3.2) {$A_3$};
\end{scope}

\end{tikzpicture}

\caption{A triple of BN 6-sided dice with winning probability $\frac{5}{9}$}\label{1stBalNontranEx}
\end{figure}
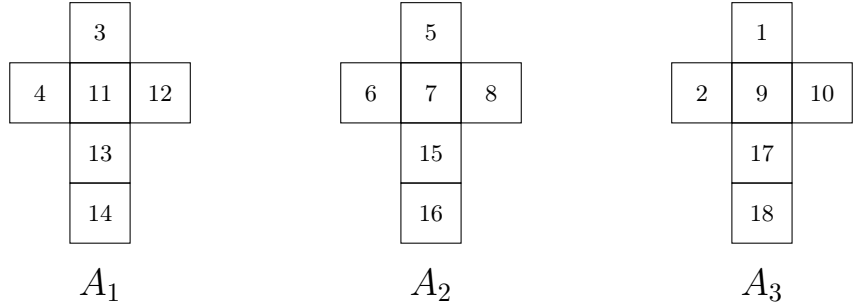

\begin{figure}[ht]
\centering
\begin{tikzpicture}[every node/.style={draw, minimum size=0.8cm, font=\footnotesize}, scale=0.8]

\node at (0,0) {10};
\node at (0,1) {1};
\node at (0,-1) {12};
\node at (-1,0) {3};
\node at (1,0) {11};
\node[draw=none, font=\Large\bfseries] at (0,-2.2) {$A_1$};

\begin{scope}[xshift=5cm]
  \node at (0,0) {5};
  \node at (0,1) {2};
  \node at (0,-1) {14};
  \node at (-1,0) {4};
  \node at (1,0) {13};
  \node[draw=none, font=\Large\bfseries] at (0,-2.2) {$A_2$};
\end{scope}

\begin{scope}[xshift=10cm]
  \node at (0,0) {8};
  \node at (0,1) {6};
  \node at (-1,0) {7};
  \node at (1,0) {9};
  \node[draw=none, font=\Large\bfseries] at (0,-2.2) {$A_3$};
\end{scope}

\end{tikzpicture}

\caption{A triple of BN dice with winning probability $\frac{3}{5}$}\label{2ndBalNontranEx}
\end{figure}
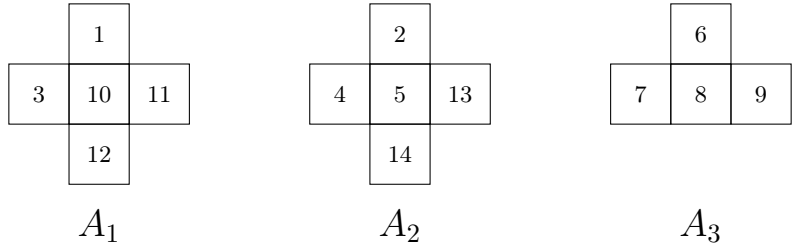

We will write $\sigma({\cal O})$ as merely $\sigma$ if the relevant ordered collection of dice is clearly implied. Naturally, we will call a word $\sigma({\cal O})$ \emph{balanced} (resp. \emph{nontransitive}) if and only if the ordered collection of dice ${\cal O}$ is \emph{balanced} (resp. \emph{nontransitive}).

Further, we will make significant use of something we will call central words. For an $n$-tuple of dice $(A_1,A_2,\dots,A_n)$, let $m_i$ be the number of sides on $A_i$, or, more precisely, the cardinality of the support of $A_i$. Then, we denote the length of the word $\sigma = \sigma(A_1,A_2,\dots,A_n)$ by $|\sigma| = m_1+m_2 + \cdots + m_n$. Now, we define a central word.

\begin{defn}
    A $(m_1,m_2,\dots,m_n)$-word $\sigma$ is \emph{central} if there exists positive integers $i,j$ with $1 \leq i < j \leq |\sigma|$ such that $\sigma(k) = A_n$ if and only if $i \leq k \leq j$. Furthermore, we say that $\sigma$ is of type $(a_1,a_2,\dots, a_{n-1})$ if the number of letter $A_\alpha$s in the first $i-1$ letters of $\sigma$ is $a_\alpha$ for $\alpha = 1,2,\dots, n-1$.
\end{defn}

For example, using the triple of dice shown in Figure \ref{2ndBalNontranEx}, the word $A_1A_2A_1A_2^2A_3^4A_1^3A_2^2$ is central of type $(2,3)$. Further, the word $A_1A_2A_3A_4A_5A_4A_3A_2A_1$ is central of type $(1,1,1,1)$.

\begin{observation} \label{SplitCentralWords}
    For a central word $\sigma$, we can write $\sigma = \tau_1 A_{n}^{m_n} \tau_2$ where $\tau_1,\tau_2$ are words for dice $(A_1,A_2,\dots,A_{n-1})$. Thus, we can write $\sigma_\text{new} = \tau_1A_n^\ell \tau_2$ for any positive integer $\ell$ and see that $\sigma_\text{new}$ is also a central word with $P_\sigma(A_i<A_{i+1}) = P_{\sigma_\text{new}}(A_i<A_{i+1})$ for $i \in \{ 1,\dots,n-1 \}$ and $P_\sigma(A_n < A_1) = P_{\sigma_\text{new}}(A_n<A_1)$.
\end{observation}
The notation $P_\sigma(A_1<A_2)$ is used to clarify that the dice $A_1$ and $A_2$ are those that belong to the word $\sigma$.

\begin{defn}
    Let $\sigma$ be a word for dice $(A_1,A_2,\dots,A_n)$. Then, we define $N_\sigma(A_\alpha < A_\beta)$ as the number of pairs $(i,j)$ with $1 \leq i < j \leq |\sigma|$ such that the $i$th letter in $\sigma$ is $A_\alpha$ and the $j$th letter in $\sigma$ is $A_\beta$.
\end{defn}

For example, using the modified Efron dice case in Figure \ref{ModEfronDice}, we see $\sigma = A_1^2A_2^3A_3^4A_4^6A_1^4A_2^3A_3^2$ so $N_\sigma(A_1<A_2) = 2\cdot 6 + 4\cdot 3 = 24$. Similarly, $N_\sigma(A_2<A_3) = 3\cdot 6 + 3 \cdot 2 = 24$. Further, note that $P_\sigma(A_\alpha < A_\beta) = \frac{N_\sigma(A_\alpha < A_\beta)}{m_\alpha \cdot m_\beta}$, where $m_i$ is the number of sides on die $A_i$ in $\sigma$.

\section{Proof of Theorem \ref{mainThm}}\label{MainProof}

We will first state lemmas \ref{mainLem1} and \ref{RationalExistenceLem} (proven in Section \ref{LemmaProofSection}). Then, we will prove Theorem \ref{mainThm} assuming these lemmas. First, we will handle the case in which $p = \pi_n$, showing that this necessitates $n$ be $4$. Then, assuming that $p < \pi_n$, we can use Lemma \ref{RationalExistenceLem} to guarantee the existence of rationals $p_2,p_3,\dots,p_k$ that satisfy the $k$ inequalities in Equation \ref{probIneqs}. We will then show that with particular choices of $s_j,m_i, \text{ and } a_\ell$, these $k$ inequalities imply the $n-2$ inequalities shown in Equation \ref{mainLem1Ineqs}, thus allowing us to apply Lemma \ref{mainLem1} to guarantee the existence of our desired BN $n$-tuple of dice.

In \cite{KKLLS2024}, D. Kim, R. Kim, Lee, Lim, and So prove the following important lemma to use in their proof of Theorem \ref{KKLLSThm}.
\begin{lemma} \label{KKLLSLemma}
    \cite{KKLLS2024} Let $m_1,m_2,m_3$ be positive integers and $a_1,a_2$ be nonnegative integers with $a_1 \leq m_1$ and $a_2 \leq m_2$. Then, for every integer $s$ such that
    $$a_1(m_2-a_2) \leq s \leq a_1m_2 + (m_1-a_1)(m_2-a_2),$$
    there exists a central $(m_1,m_2,m_3)$-word of type $(a_1,a_2)$ such that $N_\sigma(A_1<A_2) = s$.
\end{lemma}

We extend this lemma for $n \geq 4$ dice, acknowledging the extension relies on the same ideas as D. Kim, R. Kim, Lee, Lim, and So's proof of Lemma \ref{KKLLSLemma}.

\begin{lemma} \label{mainLem1}
    Let $m_1,\dots, m_n$ be positive integers and $a_1,\dots, a_{n-1}$ be nonnegative integers such that $a_i\leq m_i$ for $i=1,2,\dots, n-1$. Then, for every tuple of integers $(s_1,\dots, s_{n-2})$ such that
    \begin{equation}\label{mainLem1Ineqs}
    \begin{split}
        a_1(m_2-a_2) \leq &s_1 \leq a_1m_2 + (m_1-a_1)(m_2-a_2), \\
        a_2(m_3-a_3) \leq &s_2 \leq a_2m_3 + (m_2-a_2)(m_3-a_3), \\
        &\vdots \\
        a_{n-2}(m_{n-1}-a_{n-1}) \leq &s_{n-2} \leq a_{n-2}m_{n-1} + (m_{n-2}-a_{n-2})(m_{n-1}-a_{n-1}),
    \end{split}
    \end{equation}
    there exists a central $(m_1,m_2,\dots, m_n)$-word $\sigma$ of type $(a_1,a_2,\dots,a_{n-1})$ for dice $(A_1,A_2,\dots, A_n)$ with
    $$N_\sigma(A_i < A_{i+1}) = s_i$$
    for $i = 1,\dots,n-2$.
\end{lemma}

\begin{lemma}\label{RationalExistenceLem}
    Suppose there exists integer $n \geq 3$ and rational $p \in \left( \frac{1}{2}, \pi_n\right)$ with $\pi_n$ as in Equation \ref{Pi_nDef}. Further, suppose
    $$k = \left\lfloor \frac{n-1}{2} \right\rfloor \text{ and } \beta = \begin{cases} 1 \text{ if } n \text{ even} \\ 2 \text{ if } n \text{ odd} \end{cases}.$$
    Then, there exists rationals $p_2,p_3,\dots,p_k \in (0,1)$ such that following $k$ inequalities hold
    \begin{equation}\label{probIneqs}
    \begin{split}
        \frac{p_2}{1+p_2} \leq &p \leq \frac{1}{2-p_2} \\
        (1-p_2)p_3 \leq &p \leq 1-p_2+p_2p_3 \\
        &\vdots \\
        (1-p_{k-1})p_{k} \leq &p \leq 1-p_{k-1}+p_{k-1}p_{k} \\
        \frac{\beta}{2}(1-p_k)^\beta \leq &p \leq 1-\frac{\beta}{2}p_k^\beta
    \end{split}
    \end{equation}
\end{lemma}

Assuming Lemmas \ref{mainLem1} and \ref{RationalExistenceLem}, we will proceed to prove Theorem \ref{mainThm}.

\begin{proof}[\bf{Proof of Theorem \ref{mainThm}}]
    Suppose $n \geq 3$ is an integer and real number $p \in \left( \frac{1}{2}, \pi_n \right]$ is rational. First, consider the case in which $p = \pi_n$, implying that $\pi_n$ is rational.
    \\[5mm]Assume $\pi_n = 1-\frac{1}{4\cos^2\left( \frac{\pi}{n+2} \right)}$ is rational. Then, $\cos^2\left( \frac{\pi}{n+2} \right)$ is rational and accordingly $\cos\left( \frac{2\pi}{n+2} \right) = 2\cos^2\left( \frac{\pi}{n+2} \right)-1$ is rational. Thus, we see that $\pi_n$ rational implies that the real part of the $(n+2)$nd root of unity is also rational.  Since we are only concerned with $n\geq 3$ and only the fourth and sixth roots of unity have rational real part, $n$ must be $4$. In fact, $\pi_4 = \frac{2}{3}$.
    \\[5mm]We will now prove the $n=4$ case. Suppose that $m$ is a positive even integer sufficiently large that $pm$ is also an integer. Let $a_1 = (1-p)m,a_2=m/2, \text{ and }a_3 = pm.$ Then, since $p \in \left( \frac{1}{2}, \frac{2}{3} \right] \subseteq \left[ \frac{1}{3}, \frac{2}{3} \right]$, the following inequalities are satisfied:
    \begin{align*}
        a_1(m-a_2) \leq &pm^2 \leq a_1m + (m-a_1)(m-a_2) \text{ and} \\
        a_2(m-a_3) \leq &pm^2 \leq a_2m+(m-a_2)(m-a_3).
    \end{align*}
    Thus, by Lemma \ref{mainLem1}, there exists a central $(m,m,m,m)$-word $\sigma$ of type $(a_1,a_2,a_3)$ such that $N_\sigma(A_1<A_2) = N_\sigma(A_2<A_3) = pm^2$. Furthermore, since $\sigma$ is central of type $(a_1,a_2,a_3)$ with $a_1 = (1-p)m$ and $a_3=pm$, we have $N_\sigma(A_4<A_1) = N_\sigma(A_3<A_4) = pm^2$. Hence,
    $$P_\sigma(A_1<A_2) = P_\sigma(A_2<A_3) = P_\sigma(A_3 < A_4) = P_\sigma(A_4<A_1) = \frac{pm^2}{m^2} = p$$
    as desired.
    \\[5mm]Now, we will henceforth assume $p < \pi_n$. Let $$k = \left\lfloor \frac{n-1}{2} \right\rfloor \text{ and } \beta = \begin{cases} 1 \text{ if } n \text{ even} \\ 2 \text{ if } n \text{ odd} \end{cases}.$$
    By Lemma \ref{RationalExistenceLem}, there exists rationals $p_2,p_3,\dots,p_k \in (0,1)$ such that the following $k$ inequalities from Equation \ref{probIneqs} hold.
    \begin{align*}
        \frac{p_2}{1+p_2} \leq &p \leq \frac{1}{2-p_2} \\
        (1-p_2)p_3 \leq &p \leq 1-p_2+p_2p_3 \\
        &\vdots \\
        (1-p_{k-1})p_{k} \leq &p \leq 1-p_{k-1}+p_{k-1}p_{k} \\
        \frac{\beta}{2}(1-p_k)^\beta \leq &p \leq 1-\frac{\beta}{2}p_k^\beta
    \end{align*}
    Suppose that $m$ is a sufficiently large even integer such that $mp,mp_2,mp_3,\dots,mp_k$ are integers. Further, suppose $m_i = m$ and $s_j = pm^2$ for $i \in \{1,2,\dots,n \}$ and $j \in \{1,2,\dots,n-2 \}$. Finally, suppose $a_1,a_2,\dots,a_{n-1}$ are nonnegative integers defined as follows.
    \begin{itemize}
        \item $a_1 = (1-p)m$
        \item $a_\ell = (1-p_\ell)m$ for $\ell \in \{2,3,\dots,k \}$
        \item $a_{k+1} = \begin{cases} p_{n-(k+1)}m \text{ if } n \text{ odd} \\ \frac{m}{2} \text{ if } n \text{ even} \end{cases}$
        \item $a_\ell = p_{n-\ell}m$ for $\ell \in \{k+2,k+3,\dots,n-2 \}$
        \item $a_{n-1} = pm$
    \end{itemize}
    
    Then, with these choices of $s_j,m_i \text{ and } a_\ell$, we have the following implications:
    \begin{itemize}
        \item $\frac{p_2}{1+p_2} \leq p \leq \frac{1}{2-p_2}$ implies both
        $$\begin{cases} a_1(m_2-a_2) \leq s_1 \leq a_1m_2 + (m_1-a_1)(m_2-a_2) \\ a_{n-2}(m_{n-1}-a_{n-1}) \leq s_{n-2} \leq a_{n-2}m_{n-1} + (m_{n-2}-a_{n-2})(m_{n-1}-a_{n-1}) \end{cases}$$
        \item For $j \in \{ 2,3,\dots,k-1 \}$ and $\hat{j} = n-(j+1)$, we have $(1-p_{j})p_{j+1} \leq p \leq 1-p_{j}+p_{j}p_{j+1}$ implies both
        $$\begin{cases} a_j(m_{j+1}-a_{j+1}) \leq s_j \leq a_jm_{j+1} + (m_j-a_j)(m_{j+1}-a_{j+1}) \\ a_{\hat{j}}(m_{\hat{j}+1}-a_{\hat{j}+1}) \leq s_{\hat{j}} \leq a_{\hat{j}}m_{\hat{j}+1} + (m_{\hat{j}}-a_{\hat{j}})(m_{\hat{j}+1}-a_{\hat{j}+1}) \end{cases}$$
        \item If $n$ is even, then $\frac{\beta}{2}(1-p_k)^\beta \leq p \leq 1-\frac{\beta}{2}p_k^\beta$ is equivalent to $\frac{1-p_k}{2} \leq p \leq 1-\frac{p_k}{2}$ which implies both
        $$\begin{cases} a_k(m_{k+1}-a_{k+1}) \leq s_k \leq a_km_{k+1} + (m_k-a_k)(m_{k+1}-a_{k+1}) \\ a_{k+1}(m_{k+2}-a_{k+2}) \leq s_{k+1} \leq a_{k+1}m_{k+2} + (m_{k+1}-a_{k+1})(m_{k+2}-a_{k+2}) \end{cases}$$
        \item Finally, if $n$ is odd, then $\frac{\beta}{2}(1-p_k)^\beta \leq p \leq 1-\frac{\beta}{2}p_k^\beta$ is equivalent to $(1-p_k)^2 \leq p \leq 1-p_k^2$ which implies
        $$a_k(m_{k+1}-a_{k+1}) \leq s_k \leq a_km_{k+1} + (m_k-a_k)(m_{k+1}-a_{k+1}).$$
    \end{itemize}
    Thus, the $k$ inequalities in Equation \ref{probIneqs} imply the following $n-2$ inequalities.
    \begin{align*}
        a_1(m_2-a_2) \leq &s_1 \leq a_1m_2 + (m_1-a_1)(m_2-a_2), \\
        a_2(m_3-a_3) \leq &s_2 \leq a_2m_3 + (m_2-a_2)(m_3-a_3), \\
        &\vdots \\
        a_{n-2}(m_{n-1}-a_{n-1}) \leq &s_{n-2} \leq a_{n-2}m_{n-1} + (m_{n-2}-a_{n-2})(m_{n-1}-a_{n-1}).
    \end{align*}
    Now, we apply Lemma \ref{mainLem1} and conclude. With $m_i,s_j, \text{ and } a_\ell$ defined as above, Lemma \ref{mainLem1} guarantees the existence of a central $(m,m,\dots,m)$-word $\sigma$ of type $(a_1,a_2\dots,a_{n-1})$ with $N_\sigma(A_i < A_{i+1}) = pm^2$ for $i \in \{ 1,\dots,n-2\}$. Further, since $\sigma$ is a central word of type $(a_1,a_2\dots,a_{n-1})$ with $a_1=(1-p)m$ and $a_{n-1} = pm$, we have $N_\sigma(A_{n-1}<A_n) = pm^2$ and $N_\sigma(A_n<A_1) = pm^2$. Thus,
    $$P_\sigma(A_1<A_2) = P_\sigma(A_2<A_3) = \cdots = P_\sigma(A_n<A_1) = p,$$
    concluding that $\sigma$ is a nontransitive and balanced central $(m,m,\dots,m)$-word of type \\$(a_1,a_2,\dots,a_{n-1})$ with $w(A_1,A_2,\dots,A_n) = p$ as desired.
\end{proof}

\section{Proof of Lemmas \ref{mainLem1} and \ref{RationalExistenceLem}} \label{LemmaProofSection}

\subsection{Proof of Lemma \ref{mainLem1}}
\begin{remark}
    In this subsection, if $\sigma$ is a word for dice $(A_1,A_2,\dots,A_n)$, we will denote $q_i(\sigma) = N_\sigma(A_i<A_{i+1})$ for $i \in \{ 1,\dots, n-1\}$ and $q_n(\sigma) = N_\sigma(A_n<A_1)$.
\end{remark}

This is an extension of Lemma \ref{KKLLSLemma}. We will begin with a definition and an observation before proceeding to the proof of Lemma \ref{mainLem1}.
\begin{defn}
    The \emph{view} of a letter $A_i$ in a word $\sigma$ for dice $(A_1,A_2,\dots,A_n)$ is $\{ A_{i-1},A_{i+1} \}$ where we assume the cyclic relations $A_{n+1} := A_1$ and $A_{0} := A_n$.
\end{defn}
Note that if $A_i$ is in the view of $A_j$, then $A_j$ is also in the view of $A_i$.
\begin{observation} \label{TranspoObs}
    Let $\sigma$ be a word for dice $(A_1,A_2,\dots,A_n)$. For distinct $i,j \in \{ 1,2,\dots,n \}$, make a transposition $A_iA_j \mapsto A_jA_i$ to $\sigma$ to get $\sigma_{\text{new}}$. Then,
    \begin{itemize}
        \item Case 1: $A_j$ is not in the view of $A_i$. Then, for all $\ell \in \{ 1,2,\dots, n \}$
        $$q_\ell(\sigma) = q_\ell(\sigma_{\text{new}}).$$
        For example, $\sigma = A_3A_1A_3A_4A_3A_2A_3$ and $\sigma_\text{new} = A_1A_3^2A_4A_3A_2A_3$ with $i=3$ and $j=1$.
        \item Case 2: $j = i-1$ or ($j=n$ and $i=1$). Then, for all $\ell \in \{ 1,2,\dots,n\}\setminus \{j\}$,
        $$q_\ell(\sigma) = q_\ell(\sigma_\text{new})$$
        and $q_j(\sigma_\text{new}) = q_j(\sigma) + 1$.
        For example, $\sigma = A_3A_1A_3A_4A_3A_2A_3$ and $\sigma_{\text{new}} = A_3A_1A_3A_4A_2A_3^2$ with $i=3$ and $j=2$.
        \item Case 3: $j = i+1$ or ($j=1$ and $i=n$). Then, for all $\ell \in \{ 1,2,\dots,n\}\setminus \{i\}$,
        $$q_\ell(\sigma) = q_\ell(\sigma_\text{new})$$
        and $q_{i}(\sigma_\text{new}) = q_i(\sigma) - 1$. For example, $\sigma = A_3A_1A_3A_4A_3A_2A_3$ and $\sigma_\text{new} = A_3A_1A_3A_4A_3^2A_2$ with $i = 2$ and $j=3$.
    \end{itemize}
\end{observation}

Now, we proceed with the proof of Lemma \ref{mainLem1}.
\begin{proof}[\bf{Proof of Lemma \ref{mainLem1}}]
    The case of $n=3$, Lemma \ref{KKLLSLemma}, is proven in \cite{KKLLS2024} by D. Kim, R. Kim, Lee, Lim, and So as their Lemma 16. We will proceed by induction on $n$. Given positive integer $j \geq 3$, assume that Lemma \ref{mainLem1} holds for $n=j$. Then, we want to show the $n=j+1$ case. Since the $n=j$ case holds, we can create a central $(m_1,\dots, m_j)$-word $\tau$ of type $(a_1,a_2,\dots,a_{j-1})$ for dice $(A_1,A_2,\dots, A_j)$ with
    $$q_i(\tau) = N_\tau(A_i<A_{i+1}) = s_i$$
    for $i = 1,\dots,j-2$.
    Since $\tau$ is central, we can write
    $$\tau = \tau_1 A_j^{m_j} \tau_2$$
    where $\tau_1,\tau_2$ are words for dice $(A_1,A_2,\dots, A_{j-1})$ by Observation \ref{SplitCentralWords}. Now, construct the word
    $$\sigma_0 = A_j^{a_j} \tau_1 A_{j+1}^{m_{j+1}} A_j^{m_j-a_j} \tau_2.$$
    Note that $N_{\sigma_0}(A_i<A_{i+1}) = q_{i}(\sigma_0) = q_i(\tau) = s_i$ for $i = 1,2,\dots, j-2$ and $q_{j-1}(\sigma_0) = N_{\sigma_0}(A_{j-1}<A_j) = a_{j-1}(m_j-a_j)$.
    Using $\sigma_0$ as defined above, define $\sigma_1,\dots,\sigma_{\gamma}$ for $\gamma = a_ja_{j-1} + (m_j-a_j)(m_{j-1}-a_{j-1})$ inductively as follows.
    \begin{enumerate}
        \item Begin with $\sigma_{i+1} = \sigma_i$.
        \item Perform as many transpositions $A_{j}A_{\alpha} \mapsto A_\alpha A_j$ for $\alpha \in \{ 1,\dots,j-2\}$ as possible.
        \item Perform one transposition $A_{j}A_{j-1} \mapsto A_{j-1}A_j$.
    \end{enumerate}
    Then, by Observation \ref{TranspoObs},
    $$q_{j-1}(\sigma_{i+1}) = q_{j-1}(\sigma_i)+1$$
    and for all $\ell \in \{1,\dots,j-2\}$,
    $$q_{\ell}(\sigma_{i+1}) = q_{\ell}(\sigma_i).$$
    
    Note that step 2 will guarantee that we can transpose $A_jA_{j-1} \mapsto A_{j-1}A_j$ until we have swapped every $A_{j}$ to the left (resp. right) of $A_{j+1}^{m_{j+1}}$ with every $A_{j-1}$ to the left (resp. right) of $A_{j+1}^{m_{j+1}}$. Further, note that there are a maximum of $a_ja_{j-1}$ transpositions $A_jA_{j-1} \mapsto A_{j-1}A_j$ to the left of $A_{j+1}^{m_{j+1}}$ and $(m_j-a_j)(m_{j-1}-a_{j-1})$ to the right.
    \\[5mm] The result of this process is that $\sigma_0,\sigma_1,\dots,\sigma_\gamma$ are central words of type $(a_1,a_2,\dots,a_j)$, and for $i \in  \{0,1,\dots,\gamma\}$,
    \begin{align*}
        q_{j-1}(\sigma_i) &= a_{j-1}(m_j-a_j) + i \text{ and} \\
        q_{\ell}(\sigma_i) &= q_\ell(\sigma_0) = s_i \:\:\:\:\:\: \text{ for } \ell \in \{1,\dots, j-2 \}
    \end{align*}
    Thus, $q_{j-1}(\sigma_\gamma) = a_{j-1}m_j+(m_j-a_j)(m_{j-1}-a_{j-1})$. Then, since
    $$q_{j-1}(\sigma_0) = a_{j-1}(m_j-a_j) \leq s_{j-1} \leq a_{j-1}m_j+(m_j-a_j)(m_{j-1}-a_{j-1}) = q_{j-1}(\sigma_\gamma),$$
    we define $\sigma = \sigma_{\psi}$ for $\psi = s_{j-1} - a_{j-1}(m_j-a_j)$. That way, $\sigma$ is a central word of type $(a_1,a_2,\dots,a_{j})$ for dice $(A_1,A_2,\dots,A_{j+1})$ such that
    $$N_{\sigma}(A_\ell < A_{\ell+1}) = q_{\ell}(\sigma) = s_{\ell}$$
    for $\ell = 1,2,\dots,j-1$ as desired. We have proven the inductive step.
\end{proof}

\subsection{Proof of Lemma \ref{RationalExistenceLem}}
We will first state and prove Lemmas \ref{L(*)<1/2} and \ref{U(*)=Pi_n} which will be used in the proof of Lemma \ref{RationalExistenceLem}. Our proof of \ref{RationalExistenceLem} will rely on the existence of not necessarily rational numbers $p_2^*,p_3^*,\dots,p_k^*$, first defined in \cite{Komisarski2021}, that satisfy the inequalities seen in Equation \ref{probIneqs} from the statement of Lemma \ref{RationalExistenceLem}. Then, we will use a continuity and density argument to show that there must exist rationals $p_2,p_3,\dots,p_k$ that also satisfy these inequalities. Lemmas \ref{L(*)<1/2} and \ref{U(*)=Pi_n} show that the chosen $p_2^*,p_3^*,\dots,p_k^*$ satisfy the desired inequalities to the most desirable degree.
\begin{lemma}\label{L(*)<1/2}
    Suppose there is an integer $n \geq 3$,
    $$k = \left\lfloor \frac{n-1}{2} \right\rfloor, \text{ and } \beta = \begin{cases} 1 \text{ if } n \text{ even} \\ 2 \text{ if } n \text{ odd} \end{cases}.$$
    For $j\in \{2,3,\dots,k \}$, define
    $$p_j^* = \frac{\sin[(j+2)\alpha]}{\sin[(j+1)\alpha]\cdot 2 \cos(\alpha)} \in (0,1)$$
    where $\alpha= \frac{\pi}{n+2}$. Then,
    $$\max \left\{ \frac{p_2^*}{1+p_2^*}, (1-p_2^*)p_3^*, \dots, (1-p_{k-1}^*)p_{k}^*, \frac{\beta}{2}(1-p_k^*)^\beta \right\} < \frac{1}{2}.$$
\end{lemma}
\begin{proof}[Proof of Lemma \ref{L(*)<1/2}]
    We will begin by handling the first and last terms of the maximum. Since $p_2^* < 1$, we have $\frac{p_2^*}{1+p_2^*} < \frac{1}{2}$. If $\beta = 1$, then $\frac{\beta}{2}(1-p_k^*)^\beta < \frac{1}{2}$ because $0 < p_k^* < 1$ implies $0 < 1-p_k^* < 1$. If $\beta = 2$, then $n$ is odd so $k = \frac{n-1}{2}$. Thus,
    $$p_k^* = \frac{\sin\left( \left( \frac{n-1}{2}+2 \right) \frac{\pi}{n+2}\right)}{\sin\left( \left( \frac{n-1}{2}+1 \right) \frac{\pi}{n+2}\right) \cdot 2 \cos\left( \alpha \right)} = \frac{\sin\left( \left( \frac{n+3}{n+2} \right) \frac{\pi}{2}\right)}{\sin\left( \left( \frac{n+1}{n+2} \right) \frac{\pi}{2}\right) \cdot 2 \cos\left( \alpha \right)} = \frac{1}{2\cos\left( \alpha \right)}.$$
    Since $\cos(\alpha) < 1 < \frac{1}{2-\sqrt{2}}$,we see $1 > 2\cos(\alpha)\left( \frac{2-\sqrt{2}}{2} \right)$ which implies $\frac{1}{2\cos(\alpha)} > 1-\frac{1}{\sqrt{2}}$ since $\cos(\alpha) > 0$ for $n \geq 3$. Thus, $1-p_k^* = 1-\frac{1}{2\cos(\alpha)} < \frac{1}{\sqrt{2}}$ so $\frac{\beta}{2}(1-p_k^*)^\beta = (1-p_k^*)^2 < \frac{1}{2}$.
    \\[5mm] Now, we will show $(1-p_j^*)p_{j+1}^* < \frac{1}{2}$ for $j \in \{ 2,3,\dots, k-1 \}$. For such $j$,
    \begin{align*}
        (1-p_j^*)p_{j+1}^* &= \frac{\sin[j \alpha]}{\sin[(j+1)\alpha] \cdot 2\cos\alpha} \cdot \frac{\sin[(j+3) \alpha]}{\sin[(j+2)\alpha] \cdot 2\cos\alpha} \\
        &= \frac{1}{2} \cdot \frac{1}{2\cos^2(\alpha)} \cdot \frac{\sin[j \alpha] \cdot \sin[(j+3) \alpha]}{\sin[(j+1)\alpha] \cdot \sin[(j+2)\alpha]} \\
        &< \frac{1}{2} \cdot \frac{\sin[j \alpha] \cdot \sin[(j+3) \alpha]}{\sin[(j+1)\alpha] \cdot \sin[(j+2)\alpha]} \text{ since } \cos^2(\alpha) \geq \cos^2(\pi/5) > \frac{1}{2} \\
        &= \frac{1}{2} \cdot \frac{\left( e^{ij\alpha} - e^{-ij\alpha} \right) \cdot \left( e^{i(j+3)\alpha} - e^{-i(j+3)\alpha} \right)}{\left( e^{i(j+1)\alpha} - e^{-i(j+1)\alpha} \right) \cdot \left( e^{i(j+2)\alpha} - e^{-i(j+2)\alpha} \right)} \\
        &= \frac{1}{2} \cdot \frac{e^{i(2j+3)\alpha} + e^{-i(2j+3)\alpha} - \left( e^{i3\alpha} + e^{-i3\alpha} \right)}{e^{i(2j+3)\alpha} + e^{-i(2j+3)\alpha} - \left( e^{i\alpha} + e^{-i\alpha} \right)} \\
        &= \frac{1}{2} \cdot \frac{\cos[(2j+3)\alpha] - \cos(3\alpha)}{\cos[(2j+3)\alpha] - \cos(\alpha)}
    \end{align*}
    Thus, we consider two cases: $\cos[(2j+3)\alpha] \geq 0$ and $\cos[(2j+3)\alpha] < 0$.
    \\ {\bf Case 1:} assume $\cos[(2j+3)\alpha] \geq 0$. Then, $0 \leq \cos[(2j+3)\alpha] < \cos(3\alpha) < \cos(\alpha)$ which implies
    $$0 > \cos[(2j+3)\alpha] - \cos(3\alpha) > \cos[(2j+3)\alpha] - \cos(\alpha).$$
    \\ {\bf Case 2:} assume $\cos[(2j+3)\alpha] < 0$. Then, either $\cos[(2j+3)\alpha] < 0 \leq \cos(3\alpha) < \cos(\alpha)$ or $\cos[(2j+3)\alpha] < \cos(3\alpha) \leq 0 < \cos(\alpha)$. However, regardless,
    $$0 > \cos[(2j+3)\alpha] - \cos(3\alpha) > \cos[(2j+3)\alpha] - \cos(\alpha).$$
    \\ Therefore, in either case,
    $$\frac{\cos[(2j+3)\alpha] - \cos(3\alpha)}{\cos[(2j+3)\alpha] - \cos(\alpha)} < 1,$$
    which implies the desired $(1-p_j^*)p_{j+1}^* < \frac{1}{2}$.
\end{proof}

\begin{lemma}\label{U(*)=Pi_n}
    Suppose
    $$k = \left\lfloor \frac{n-1}{2} \right\rfloor \text{ and } \beta = \begin{cases} 1 \text{ if } n \text{ even} \\ 2 \text{ if } n \text{ odd} \end{cases}.$$
    For $j\in \{2,3,\dots,k \}$, define
    $$p_j^* = \frac{\sin[(j+2)\alpha]}{\sin[(j+1)\alpha]\cdot 2 \cos(\alpha)}$$
    where $\alpha= \frac{\pi}{n+2}$. Then,
    $$\min \left\{ \frac{1}{2-p_2^*},1-p_2^*+p_2^*p_3^*, \dots, 1-p_{k-1}^*+p_{k-1}^*p_{k}^*, 1-\frac{\beta}{2}(p_k^*)^\beta \right\} = \pi_n$$
    where $\pi_n = 1-\frac{1}{4\cos^2(\alpha)}$ as seen in equation (\ref{Pi_nDef}).
\end{lemma}
\begin{proof}[Proof of Lemma \ref{U(*)=Pi_n}]
    We will begin by handling the first term in the minimum. Notice
    $$\frac{1}{2-p_2^*} = \frac{2\sin(3\alpha)\cos(\alpha)}{4\sin(3\alpha)\cos(\alpha) - \sin(4\alpha)}.$$
    Recalling $\sin(4\alpha) = 4\sin(\alpha)\cos(\alpha)\left( \cos^2(\alpha) - \sin^2(\alpha) \right)$ and $\sin(3\alpha) = 3\sin(\alpha) - 4\sin^3(\alpha)$, we have
    \begin{align*}
        \frac{1}{2-p_2^*} &= \frac{2\sin(\alpha)\cos(\alpha) \left(3-4\sin^2(\alpha) \right)}{4\sin(\alpha)\cos(\alpha) \left( 3-3\sin^2(\alpha)-\cos^2(\alpha) \right)} \\
        &= \frac{3-4\sin^2(\alpha)}{2\left( 3 - \left(\sin^2(\alpha)+ \cos^2(\alpha) \right) -2\sin^2(\alpha) \right)} \\
        &= 1 - \frac{1}{4-4\sin^2(\alpha)} = 1-\frac{1}{4\cos^2(\alpha)} \\
        &= \pi_n
    \end{align*}
    as desired.
    \\ Now, notice
    $$1-p_j^* = \frac{\sin(j\alpha)}{\sin\left[ (j+1)\alpha \right] \cdot 2\cos(\alpha)}.$$
    We will use this to show $1-p_j^*+p_j^*p_{j+1}^* = \pi_n$ for $j \in \{ 2,3,\dots,k-1 \}$. For such $j$,
    $$1-p_j^*(1-p_{j+1}^*) = 1-\left( \frac{\sin\left[ (j+2)\alpha \right]}{\sin\left[ (j+1)\alpha \right] \cdot 2\cos\alpha} \right) \left( \frac{\sin[(j+1)\alpha]}{\sin[(j+2)\alpha] \cdot 2\cos \alpha} \right) = 1-\frac{1}{4\cos^2(\alpha)} = \pi_n$$
    as desired.
    \\ We will finally handle the last term in the minimum. If $\beta = 1$, then $n$ is even so $k = \frac{n-2}{2}$. Thus,
    $$p_k^* = \frac{\sin\left( \left(\frac{n+2}{2} \right) \alpha \right)}{2\sin \left( \frac{n}{2} \alpha \right) \cos(\alpha)} = \frac{\sin\left( \frac{\pi}{2} \right)}{2\sin \left( \frac{n\pi}{2(n+2)} \right)\cos(\alpha)} = \frac{1}{2\sin\left( \frac{\pi}{2} - \frac{\pi}{n+2} \right) \cos(\alpha)} = \frac{1}{2\cos^2(\alpha)},$$
    so
    $$1-\frac{\beta}{2}(p_k^*)^\beta = 1-\frac{1}{2}p_k^* = 1- \frac{1}{4\cos^2(\alpha)} = \pi_n.$$
    If $\beta = 2$, then $n$ is odd so $k = \frac{n-1}{2}$. Thus,
    $$p_k^* = \frac{\sin\left( \frac{n+3}{2} \alpha\right)}{2\sin\left( \frac{n+1}{2} \alpha \right) \cos(\alpha)} = \frac{\sin\left( \frac{n+3}{n+2} \cdot \frac{\pi}{2} \right)}{2\sin\left( \frac{n+1}{n+2} \cdot \frac{\pi}{2} \right) \cos(\alpha)} = \frac{1}{2\cos(\alpha)},$$
    so
    $$1-\frac{\beta}{2}(p_k^*)^\beta = 1-(p_k^*)^2 = 1-\frac{1}{4\cos^2(\alpha)} = \pi_n$$
    as desired.
\end{proof}

Now that we have proven Lemmas \ref{L(*)<1/2} and \ref{U(*)=Pi_n}, we will proceed with the proof of Lemma \ref{RationalExistenceLem}.
\begin{proof}[\bf{Proof of Lemma \ref{RationalExistenceLem}}]
    It suffices to show existence of rationals $p_2,\dots, p_k \in (0,1)$ such that
    $$L(p_2,\dots,p_k) \leq p \leq U(p_2,\dots,p_k)$$
    for $L,U: [(0,1)\cap \mathbb{Q}]^{k-1} \to \mathbb{R}$ defined as
    \begin{align*}
        L(x_2,x_3,\dots,x_k) &= \max \left\{ \frac{x_2}{1+x_2}, (1-x_2)x_3, \dots, (1-x_{k-1})x_{k}, \frac{\beta}{2}(1-x_k)^\beta \right\} \text{ and}\\
        U(x_2,x_3,\dots,x_k) &= \min \left\{ \frac{1}{2-x_2},1-x_2+x_2x_3, \dots, 1-x_{k-1}+x_{k-1}x_{k}, 1-\frac{\beta}{2}x_k^\beta \right\}.
    \end{align*}
    \\[5mm] First, an overview of how we will prove this: We will define not necessarily rational $p_2^*, p_3^*,\dots, p_k^* \in [0,1]$ such that $L(p_2^*,\dots, p_k^*) < 1/2$ and $U(p_2^*,\dots, p_k^*) = \pi_n$. Then, using continuity of $L,U$ and density of the rationals in the reals, we will guarantee existence of $(p_2,\dots,p_k)$ sufficiently close to $(p_2^*,\dots, p_k^*)$ such that
    $$L(p_2,\dots,p_k) \leq p \leq U(p_2,\dots,p_k).$$
    \\[5mm] Now, we will prove existence of such rationals. To begin, we define $p_j^*$ as is done in the proof of Theorem \ref{Komisarski} by Komisarski \cite{Komisarski2021}. For $j \in \{ 2,3,\dots, k\}$, define
    $$p_j^* = \frac{\sin\left[ (j+2)\alpha \right]}{\sin\left[ (j+1)\alpha \right] \cdot 2\cos(\alpha)}$$
    where $\alpha = \frac{\pi}{n+2}$. Further, define $\Gamma_n = L(p_2^*,\dots,p_k^*)$ to see
    \begin{align*}
        \Gamma_n &= L(p_2^*,\dots,p_k^*) < 1/2 \text{ and} \\
        \pi_n &= U(p_2^*,\dots,p_k^*)
    \end{align*}
    by Lemmas \ref{L(*)<1/2} and \ref{U(*)=Pi_n}. Now, define
    $$\varepsilon = \min\left\{ \frac{\pi_n-p}{2}, \frac{1/2 - \Gamma_n}{2} \right\} > 0.\footnote{The reason we had to handle the $p=\pi_n$ case separately in the proof of Theorem \ref{mainThm} was to ensure that this $\varepsilon$ is strictly positive.}$$
    That way,
    $$L(p_2^*,\dots,p_k^*) + \varepsilon \leq p \leq U(p_2^*,\dots,p_k^*)-\varepsilon.$$
    Then, by continuity of $L$ and $U$, there exists $\delta > 0$ such that, using the euclidean norm,
    \\$\|(x_2,\dots,x_k)-(p_2^*,\dots,p_k^*)\| < \delta$ implies
    \begin{align*}
        |L(x_2,\dots,x_k) - L(p_2^*,\dots,p_k^*)| &< \varepsilon \text{ and} \\
        |U(x_2,\dots,x_k) - U(p_2^*,\dots,p_k^*)| &< \varepsilon.
    \end{align*}
    Recalling that the rationals are dense in the reals and subsequently $\mathbb{Q}^{k-1} \subseteq \mathbb{R}^{k-1}$ is dense, there exists rationals $p_2,\dots,p_k$ such that $\|(p_2,\dots,p_k) - (p_2^*,\dots,p_k^*)\| < \delta$. Consequently,
    \begin{align*}
        |L(p_2,\dots,p_k) - L(p_2^*,\dots,p_k^*)| &< \varepsilon \text{ and} \\
        |U(p_2,\dots,p_k) - U(p_2^*,\dots,p_k^*)| &< \varepsilon.
    \end{align*}
    Thus,
    $$L(p_2,\dots,p_k) < L(p_2^*,\dots,p_k^*) + \varepsilon < p \leq U(p_2^*,\dots,p_k^*) - \varepsilon < U(p_2,\dots,p_k)$$
    as desired.
\end{proof}

\section{Potential Future Work}

Although we have resolved Problem \ref{KKLLSProb}, one is yet to resolve the following problem that D. Kim, R. Kim, Lee, Lim, and So \cite{KKLLS2024} posed.
\begin{problem}\label{futureProb}
    \cite{KKLLS2024} Find all $n$-tuples $(m_1,m_2,\dots,m_n)$ of positive integers such that there exists a BN $n$-tuple of dice $(A_1,A_2,\dots,A_n)$ such that $A_i$ has $m_i$ sides.
\end{problem}
In \cite{SchaeferSolo2017,SchaeferSchweig2017}, Schaefer and Schweig found that for positive integers $n,m \geq 3$, there exists a BN $n$-tuple of dice, each with $m$ sides. In \cite{KKLLS2024}, the $n=3$ case of Problem \ref{futureProb} was fully resolved. Further, in April 2025, a preprint by R. Kim \cite{Kim2025Balanced4Dice} resolves the $n=4$ case. The proofs for the $n=3$ and $n=4$ cases, in \cite{KKLLS2024} and \cite{Kim2025Balanced4Dice} respectively, rely on the least upper bound of the winning probability. Accordingly, after our result, the extension to $n\geq 5$ in Problem \ref{futureProb} may be easier to resolve using similar techniques as the $n=3$ and $n=4$ cases.

\section{Acknowledgements}
This work was supported by the Harvard Center of Mathematical Sciences and Applications, the Harvard College Research Program Summer 2024, and the Harvard University Department of Mathematics. This project heavily benefited from advisement by Dr. Philip Matchett Wood, Dr. Michael Nathanson, Dingding Dong, and Lauren Teichholtz.

\bibliographystyle{plain}
\bibliography{biblio.bib}

\end{document}